\documentclass{article}

\usepackage{amsmath,amsfonts,amsthm,amssymb}
\usepackage{enumerate}
\usepackage{graphicx}
\usepackage{color}
\usepackage{tabularx}

\newcommand{\Q}{\mathbb{Q}}
\newcommand{\var}[1]{\mathcal{#1}}
\newcommand{\alphabet}[1]{\mathcal{#1}}
\newcommand{\alg}[1]{\mathbf{#1}}

\newcommand{\Con}{\mathbf{Con}}

\theoremstyle{plain}
\newtheorem{theorem}{Theorem}[section]
\newtheorem{corollary}{Corollary}[section]
\newtheorem{question}{Question}[section]
\newtheorem{proposition}{Proposition}[section]

\newtheorem{claim}[theorem]{Claim}
\newtheorem{lemma}[theorem]{Lemma}
\def\indexed #1#2{\mathop{#1\kern0pt}\limits_{#2}}

\def\xmatrix#1#2{\xmatrixA{\left(\vphantom{\begin{matrix}#2\end{matrix}}\right.}%
   \begin{matrix}#1\cr#2\end{matrix}%
   \xmatrixA{\left.\vphantom{\begin{matrix}#2\end{matrix}}\right)}%
}
\def\xmatrixA#1{\lower1.5ex\hbox{$\vcenter{\hbox{$#1$}}$}}

\title{Maltsev conditions for general congruence meet-semidistributive
  algebras}
\author{Miroslav Ol\v s\'ak}

\begin{document}

\maketitle

\thanks{
  Partially supported
  by the Czech Grant Agency (GAČR) under grant no. 18-20123S,
  by the National Science Centre Poland under grant no. UMO-2014/13/B/ST6/01812
  and by the PRIMUS/SCI/12 project of the Charles University.
}
\begin{abstract}
  Meet semidistributive varieties are in a sense the last of
  the most important classes in universal algebra
  for which it is unknown whether it can be characterized by a strong
  Maltsev condition.
  We present a new, relatively simple Maltsev condition characterizing
  the meet-semidistributive varieties, and provide a candidate for a
  strong Maltsev condition.
\end{abstract}

\section{Introduction}

The tame congruence theory (TCT)~\cite{FiniteAlgebras}, a structure
theory of general finite algebras, has revealed that there are only 5
possibly local behaviors of a finite algebra:
\begin{enumerate}[(1)]
  \itemsep0em 
\item algebra having only unary functions, 
\item one-dimensional vector space,        
\item the two-element boolean algebra,     
\item the two-element lattice,             
\item the two element semilattice.         
\end{enumerate}

If there is a local behavior of type (i) in an algebra $\alg A$,
the algebra is said to \emph{have} type (i).
A $\var V$ variety have type (i) if there is an algebra
$\alg A\in\var V$ that have (i).
If an algebra or variety does not have a type (i), it is said to
\d{omit} type (i).
The set of ``bad'' types that are omitted in a variety is an important
structural information; for instance, it plays a significant role in
the fixed-template constraint satisfaction
problem~\cite{CSPSurvey}. The ``worst'' type is type (1) and omitting
it has been characterized in many equivalent ways, one of which is
given in the following theorem.

\begin{theorem}
  \label{wnu-theorem}
\cite{WNU} A locally finite variety $\var V$ omits type (1) if and
only if there is an idempotent WNU (weak near unanimity) term in $\alg
A$, that is a term satisfying the following identities:
\begin{itemize}
  \item idempotence: $t(x,x,x,\ldots,x)=x$,
  \item weak near unanimity: $t(y,x,x,\ldots,x) = t(x,y,x,\ldots,x) =
    \cdots = t(x,\ldots,x,y)$
\end{itemize}
for any $x,y\in\alg A$.
\end{theorem}

Such a characterization of
varieties of algebras by means of the existence of terms satisfying
certain identities are in general called Maltsev conditions. More
precisely, a strong Maltsev condition is given by a finite set of term
symbols and a finite set of
identities. A given strong Maltsev condition is satisfied in a variety
$\var V$ if we can substitute the term symbols by actual terms in the
variety in such a way that all the identities are satisfied. A general
Maltsev condition is then a disjuction of countably many strong Maltsev
conditions (as in the example of Theorem~\ref{wnu-theorem}).

Whenever a variety $\var V$ satisfies a certain Maltsev condition and
$\var W$ is interpretable into $\var W$, then $\var W$ satisfies the
Maltsev condition too. For the notion of interpretability, we refer
the reader to~\cite{FiniteAlgebras}. There are following relations
between types of locally finite varieties and the interpretability.
\begin{itemize}
\item Any variety that has type (1) is interpretable into any variety.
\item Any variety is interpretable into a variety that has type (3).
\item Any variety that has type (5) is interpretable into a variety
  that has type (4).
\end{itemize}
Therefore, it is reasonable to ask for the Maltsev
conditions for the following classes:
$$
\var M_{\{1\}}, \var M_{\{1,2\}}, \var M_{\{1,5\}},
\var M_{\{1,2,5\}},\var M_{\{1,4,5\}}, \var M_{\{1,2,4,5\}},
$$
where $\var M_S$ is the class of all the algebras that omits all the types
from the set $S$. There is an appropriate Maltsev condition
for all six classes.

It was proved that $\var M_{\{1\}}$ and $M_{\{1,2\}}$ can be
characterized by strong Maltsev conditions. Recall that idempotent
term is a term $t$ satisfying the equation $t(x,x,\ldots,x)=x$.
\begin{theorem}
  \label{siggers}
\cite{OptimalStrong} A locally finite variety omits type (1) if and
only if it has an idempotent 4-ary term $s$ satisfying $s(r,a,r,e) =
s(a,r,e,a)$.
\end{theorem}
\begin{theorem}[Theorem 2.8 of \cite{Kozik2015}]
  \label{wnu-glued}
  A locally finite variety omits types (1) and (2) if and only if it has
  three-ary and four-ary idempotent terms $w_3,w_4$ satisfying equations
  $$
  w_3(yxx)=w_3(xyx)=w_3(xxy)=w_4(yxxx)
  $$$$
  =w_4(xyxx)=w_4(xxyx)=w_4(xxxy).
  $$
\end{theorem}

In the same paper~\cite{Kozik2015} the authors have demostrated that
the remaining classes, that is $\var M_{{1,5}}, \var M_{{1,2,5}}, \var
M_{{1,4,5}}, \var M_{{1,2,4,5}}$,
cannot be characterized by strong Maltsev conditions.

Although types in the TCT are defined only for locally finite
varieties (because only finite algebras are assigned types), the
type-omitting classes have alternative characterizations which do not
refer to the type-set. They are shown in the following table taken
from~\cite{Kozik2015}.

\begin{center}
  \begin{tabularx}{\textwidth}{ | l | X | }
    \hline
    Type Omitting Class        &  Equivalent property,\cr
\hline
\hline
    $\var M_{\{1\}}$              &  satisfies a nontrivial idempotent Maltsev condition,\cr
\hline
    $\var M_{\{1,5\}}$            &  satisfies a nontrivial congruence identity,\cr
\hline
    $\var M_{\{1,4,5\}}$          &  congruence $n$-permutable, for some $n>1$,\cr
\hline
    $\var M_{\{1,2\}}$            &  congruence meet semidistributive,\cr
\hline
    $\var M_{\{1,2,5\}}$          &  congruence join semidistributive,\cr
\hline
    $\var M_{\{1,2,4,5\}}$        &  congruence $n$-permutable for some $n$ and
                              congruence join semidistributive. \cr
\hline
  \end{tabularx}
\end{center}

Each of the properties in the right column of the table is
characterized by an idempotent Maltsev
condition~\cite{FiniteAlgebras} for 
general (not necessarily locally finite) varieties. However,
Theorems~\ref{siggers}, \ref{wnu-glued} giving strong Maltsev
conditions are not guaranteed to 
work. Indeed, there is an example of an
idempotent algebra that satisfy a non-trivial Maltsev
condition, but has no term $s(r,a,r,e)=s(a,r,e,a)$,
see~\cite{Kazda}. However, it turned out that
the first property is characterized by another strong Maltsev
condition.

\begin{theorem}
  \cite{WeakestIdempotent}
  An idempotent algebra satisfy a non-trivial Maltsev condition if and
  only if it has a term $t$ such that
  $$
  t(yxx,xyy) = t(xyx,yxy) = t(xxy,yyx).
  $$
\end{theorem}

The finite counterexamples to strong Maltsev conditions for
$$\var M_{\{1,5\}}, \var M_{\{1,2,5\}},\var M_{\{1,4,5\}},
\var M_{\{1,2,4,5\}}$$
work as counterexamples for the general case, so the remaining
question is the following.

\begin{question}
  \label{main-question}
  Is there a strong Maltsev condition that is equivalent to congruence
  meet-semidistributivity?
\end{question}

\subsection{Congruence meet-semidistributivity}

By $\Con(\alg A)$ we denote the lattice of congruences of $\alg A$.
A variety $\var V$ is said to be congruence meet-semidistributive
(shortly $SD(\wedge)$) if
for any $\alg A\in\var V$, and any three congruences
$\alpha,\beta,\gamma\in\Con(\alg A)$ such that
$$
\alpha\wedge\gamma = \beta\wedge\gamma,
$$
we have
$$
\alpha\wedge\gamma = \beta\wedge\gamma = (\alpha\vee\beta)\wedge\gamma.
$$

This property has many equivalent definitions,
see Theorem~8.1 in~\cite{CongruenceLattices}, we mention some
of them.
\begin{theorem}
  Let $\var V$ be a variety. The following are equivalent.
  \begin{itemize}
  \item $\var V$ is a congruence meet-semidistributive variety.
  \item No member of $\var V$ has a non-trivial abelian congruence.
  \item $[\alpha,\beta]=\alpha\wedge\beta$ for all
    $\alpha,\beta\in\Con(\alg A)$ and all $\alg A\in\var V$, where
    $[\alpha,\beta]$ denotes the commutator of congruences.
  \item The diamond lattice $M_3$ is not embeddable in $\Con(\alg A)$
    for any $\alg A\in \var V$,
  \item $\var V$ satisfies an idempotent Maltsev condition that fails
    in any finite one-dimensional vector space over a non-trivial
    field (equivalently in any module).
  \end{itemize}
\end{theorem}

In this paper we are going to study the Maltsev conditions satisfied
by every $SD(\wedge)$ variety. Not only is it not known whether
there is a strong Maltsev condition characterizing the $SD(\wedge)$
varieties, but the known Maltsev conditions for $SD(\wedge)$ were quite
complicated. Probably the simplest Maltsev condition for $SD(\wedge)$
which was available before this work is the following one.

Let $[n]$ denote the set $\{1,2,\ldots,n\}$.
Consider some $n$, and a self-inverse bijection
$\varphi\colon[2n]\to[2n]$ without fixed points, such that whenever
$i < j < \varphi(i)$, then also $i < \varphi(j) < \varphi(i)$. Such a
bijection corresponds to a proper bracketing sequence with $n$
opening and $n$ closing brackets. Then the \emph{bracket terms}
are ternary terms $b_1,\ldots,b_{2n}$ satisfying the following
identities
$$
b_1(x,y,z) = x, \quad b_{2n}(x,y,z) = z,
$$$$
b_{2i}(y,x,x) = b_{2i-1}(y,x,x), \quad b_{2i}(x,x,y) = b_{2i+1}(x,x,y),
$$$$
b_{i}(x,y,x) = b_{\varphi(i)}(x,y,x),
$$
for any $i$ where it makes sense.

\begin{theorem}[Theorem 1 in \cite{BracketTerms}]
  \label{bracket-terms}
  A variety $\var V$ satisfies the $SD(\wedge)$ property if and only
  if it has some bracket terms.
\end{theorem}

\subsection{The new terms}

In this paper we define $(m_1+m_2)$-terms as a triple of idempotent terms
$(f, g_1, g_2)$, where $g_1$ is $m_1$-ary, $g_2$ is $m_2$-ary, $f$ is
$(m_1+m_2)$-ary, and they satisfy the identities
\begin{align*}
f(x,x,\ldots,x,\indexed y i,x,\ldots,x) &= g_1(x,x,\ldots,x,\indexed y i,x,\ldots,x)
\hbox{ for any $i=1,\ldots,m_1$},\cr
f(x,x,\ldots,x,\indexed y{n_1+i},x,\ldots,x) &= g_2(x,x,\ldots,x,\indexed y i,x,\ldots,x)
\hbox{ for any $i=1,\ldots,m_2$}.\cr
\end{align*}

We prove the following theorem.
\begin{theorem}
  A variety $\var V$ is congruence meet-semidistributive if and only if it has
  $(3+m)$-terms for some $m$.
\end{theorem}
Checking the backward implication is easy. For a contradiction, assume
that the identities of $(m_1+m_2)$-terms were satisfied in modules. That
means that $f$, $g_1$, $g_2$ are represented by linear
combinations. In particular, let
$$
f(x_1,x_2,\ldots,x_{m_1+m_2}) = a_1x_1+\cdot+a_{m_1+m_2}x_{m_1+m_2}
$$$$
g_1(x_1,x_2,\ldots,x_{m_1}) = b_1x_1+\cdot+b_{m_1}x_{m_1},
\quad g_2(x_1,x_2,\ldots,x_{m_2}) = c_1x_1+\cdot+c_{m_2}x_{m_2},
$$
By plugging $x=0, y\neq0$ into the identities for $f$ and $g_1$, we get
$a_i=b_i$ for $i=1,\ldots,m_1$. If we make the same substitution
in the second identity, we get $a_{m_1+i}=c_i$ for
$i=1,\ldots,m_2$. Moreover, idempotency identity enforces
$$
\sum_{i=1}^{m_1+m_2} a_i = \sum_{i=1}^{m_1} b_i = \sum_{i=1}^{m_2} c_i
= 1.
$$
Therefore we get
$$
1 = \sum_{i=1}^{m_1+m_2} a_i = \sum_{i=1}^{m_1} b_i + \sum_{i=1}^{m_2} c_i
= 2,
$$
which contradicts that our field was non-trivial.
Thus, we proved the backward implication.

To prove the forward implication, we take a detour through a
generalized version of $(m_1+m_2)$-terms. Given $n,m$, we define
$n\times(n+1)\times m$-terms as follows.

Let $i$ have values from $1$ to $n$, $j$ have values from
$1$ to $n+1$, and $k$ have values from $1$ to $m$.
The $n\times(n+1)\times m$-terms are idempotent $(n+1)m$-ary terms
$f_i$ (variables are indexed by pairs $(j,k)$) and idempotent $nm$-ary terms
$g_j$ (variables are indexed by pairs $(i,k)$) such that for every
$i,j,k$ they satisfy the equation
$$
f_i(x,x,\ldots,x,\indexed y{(j,k)},x,\ldots,x)
= g_j(x,x,\ldots,x,\indexed y{(i,k)},x,\ldots,x).
$$

By definition, $1\times2\times m$-terms are equivalent to the
$(m+m)$-terms. On the other hand, for large enough $n,m$, it is simple
to derive the $n\times(n+1)\times m$-terms from another Maltsev
condition not satisfiable in vector spaces.

\begin{proposition}
  \label{sd-meet-weakest}
  Let $\var V$ be a $SD(\wedge)$ variety. Then $\var V$ has
  $n\times(n+1)\times m$-terms for some $n,m$.
\end{proposition}
\begin{proof}
  By Theorem~\ref{bracket-terms}, we may assume that there are bracket
  terms $b_1,\ldots,b_{2n}$ corresponding to a bijection
  $\varphi\colon [2n]\to[2n]$. Notice that since $\varphi$ forms a proper
  bracketing, $\varphi(i)$ has a different parity than $i$ for any
  $i$. Let $\psi(i) = \varphi(2i-1)/2$ and $\psi'(i) =
  (\varphi(2i)+1)/2$. In other words, we splited $[2n]$ to odd and even
  part and labeled them as $[n]$; then $\psi$ corresponds to the
  mapping $\varphi$ odd${}\to{}$even, and $\psi'$ to its inverse.
  We construct $n\times (n+1)\times 3$-terms as follows. We set
  \begin{align*}
  g_1(x_{1,1},\ldots,x_{\varphi(1),2},\ldots) &= x_{1,1} = b_1(x_{1,1},x_{\psi(1),2},x),\cr
  g_i(\ldots,x_{i,1},\ldots,x_{\psi(i),2},\ldots,x_{i-1,3},\ldots)
  &= b_{2i-1}(x_{i,1},x_{\psi(i),2},x_{i-1,3})\cr
  g_{n+1}(\ldots,x_{n,3}) &= x_{n,3} \cr
  f_{i}(\ldots,x_{i,1},\ldots,x_{\psi'(i),2},\ldots,x_{i+1,3},\ldots)
  &= b_{2i}(x_{i,1},x_{\psi'(i),2},x_{i+1,3})\cr
  \end{align*}
  All the $n\times(n+1)\times3$-identities follows directly from the
  bracket identities.
\end{proof}

\subsection{Outline}

The rest of the proof is divided into two sections. In
Section~\ref{semirings} we show that in $n\times(n+1)\times m$-terms,
we can decrease $n$ by one increasing $m$ enough. It follows that any
$SD(\wedge)$ variety has $(m+m)$-terms a large enough
$m$. In Section~\ref{3-plus-m}, we improve that result to
$(3+m)$-terms. Section~\ref{counterexample} then provides a few
counterexamples showing that requesting $(2+m)$-terms would be too
strong. Finally, in Section~\ref{further-work} we discuss remaining
open questions.

\section{Simplifying $n\times(n+1)\times m$-terms}
\label{semirings}

\subsection{Semirings}

We will need some basic facts about semirings for our first proof.

Semiring is a general algebra $\alg A=(A, +, \cdot, 0, 1)$ where
$(A,+,0)$ is a commutative monoid, $(A, \cdot, 1)$ is a monoid, zero
absorbes everything in multiplication ($0\cdot x = x\cdot 0 = 0$), and
distributive laws are satisfied, that is,
$a\cdot(b+c)=a\cdot b+a\cdot c$ and
$(a+b)\cdot c = a\cdot c + b\cdot c$. As usual, the binary
multiplication operation $\cdot$ is often ommited writing $ab$ instead
of $a\cdot b$.

Let $\alphabet A$ be an alphabet. The elements of the free monoid
$\alphabet A^*$ generated by $\alphabet A$ are represented by finite words in
the alphabet, multiplication concatenates the words and the constant 1
corresponds to the empty word. Finally, the elements of the free
semiring generated by $\alphabet A$ are represented as finite multisets
(formal sums) of words in $\alphabet A^*$. The addition in the free semiring
is defined as sums (disjoint unions) of the corresponding multisets,
and the product $p\cdot q$ is defined as piecewise product of the
monomials, that is $\{u\cdot v : u\in p, v\in q\}$.

Let $\alg F$ be the free semiring generated by some alphabet $\alphabet A$,
and $E$ be a set of equations of the form
$e_1= 1, e_2= 1, e_3= 1, \ldots$ where
$e_i\in\alg F$. We are going to provide a description of the conguence
on $\alg F$ generated by $E$.

Take a monomial $u\in\alphabet A^*$. By a \d{single expansion} of $u$ we
mean any element of $\alg F$ of the form $ve_iw$ where $vw=u$.
A \d{single expansion} on a general element of $\alg F$ is then
defined as performing a single expansion on one of its summands.
Finally, we say that $p$ is an \d{expansion} of $q$ if we can obtain
$p$ by performing consecutive single expansion steps on $q$.

\begin{proposition}
\label{common-expansion}
For any pair $(p,q)$ of elements in $\alg F$, these two elements are
congruent modulo the congruence generated by $E$ if and only if there
is a common expansion $r$ of both $p$ and $q$.
\end{proposition}
\begin{proof}
The backward implication is obvious: If $r$ is an expansion of $p$,
then $r$ is clearly congruent to $p$. Analogously, $r$ is congruent to
$q$, therefore $p$ is congruent to $q$. We are going to prove the
forward implication.

For $p,q\in\alg F$ we define a relation $p\sim q$ if there is a common
expansion of $p$ and $q$. Clearly each $e_i\sim 1$. To show that
$\sim$ includes the congruence generated by $E$, it remains to prove
that $\sim$ is a congruence. Symmetry and reflexivity is apparently
satisfied, so we have to prove that $\sim$ is transitive and compatible
with the operations. To do that, let us introduce some notation.

Let $p \leq q$ denote that $q$ is an expansion of $p$ and
let $p\preccurlyeq q$ denote that $q$ can be obtained by applying single
expansion steps on a subset of summands of $p$. So $p\preccurlyeq q$ is
stronger than $p\leq q$ but weaker that $q$ being a single expansion
of $p$.

These orderings are clearly closed under addition. In particular, if
$p=\sum_i^n p_i$, $q=\sum_i^n q_i$ and
$p_i\preccurlyeq q_i$, then $p\preccurlyeq q$.

\begin{claim}
\label{prec-mul-comp}
For any $p,q,r,s\in\alg F$ such that $p\preccurlyeq q$ we have
$rps\preccurlyeq rqs$.
\end{claim}
To verify that, let
$p=\sum_i^P p_i$, $q = \sum_i^P q_i$,
$r=\sum_i^R r_i$, $s=\sum_i^S s_i$,
where $p_i, r_i, s_i$ are monomials and $p_i \preccurlyeq q_i$.
Then
$$
rps = \sum_i^R\sum_j^P\sum_k^S r_ip_jk_k, \quad
rqs = \sum_i^R\sum_j^P\sum_k^S r_iq_js_k.
$$
Since $p_j\preccurlyeq q_j$,
we can write $p_j = u_jv_j$ so that $q_j = u_jx_jv_j$ where
$x_j\succcurlyeq 1$, that is, $x=1$ or $x$ one of the elements $e_i$.
So we can
write $r_ip_js_k = (r_iu_j)(v_js_k)$ and $r_iq_j = (r_iu_j)x_j(v_js_k)$.
Therefore $r_ip_js_k\preccurlyeq r_iq_js_k$ and
thus $rps\preccurlyeq rqs$.

\begin{claim}
\label{prec-diamond}
For any $p,q,r\in\alg F$ such that $r\preccurlyeq p$
and $r\preccurlyeq q$ there exists $s\in\alg F$ such that
$p\preccurlyeq s$ and $q\preccurlyeq s$.
\end{claim}
First, we prove the claim if $r$ is a monomial. So polynomials $p$, $q$
are constructed by inserting $p'$, $q'$ somewhere into $r$
respectively, where $p', q'\succcurlyeq 1$.
Without loss of generality, $q'$ is inserted at the
same position as $p'$ or later, so we can write $r=uvw$,
$p=up'vw$, $q=uvq'w$. Now we choose $s=up'vq'w$. By
Claim~\ref{prec-mul-comp} and $p', q'\succcurlyeq 1$ we get
the required
$$
p = (up'v)(w) \preccurlyeq (up'v)q'(w) = s, \quad
q = (u)(vq'w) \preccurlyeq (u)p'(vq'w) = s.
$$
For a general $r=\sum_i^n r_i$ where $r_i$ are monomials, we decompose
$p=\sum_i^n p_i$, $q=\sum_i^n q_i$ so that
$r_i\preccurlyeq p_i, q_i$. Therefore, we find elements $s_i$ such
that $s_i\succcurlyeq p_i, q_i$, and eventually
$s=\sum_i^n\succcurlyeq p, q$.

We are finally ready to prove the transitivity of $\sim$ and compatibility
with operations.
\begin{claim}
If $x,r,y\in\alg F$, $x\sim r$ and $r\sim y$, then $x\sim y$.
\end{claim}
By definition of $\sim$, there are $p,q\in\alg F$ such that
$x, r\leq p$ and $r,y\leq q$. We break
the expansion $r\leq p$ into finite number of single expansion steps
getting a sequence
$$
r = s_{0,0}\preccurlyeq s_{1,0} \preccurlyeq
\cdots \preccurlyeq s_{P,0} = p.
$$
Similarly, there is a sequence
$$
r = s_{0,0}\preccurlyeq s_{0,1} \preccurlyeq
\cdots \preccurlyeq s_{0,Q} = q.
$$
By repeated application of Claim~\ref{prec-diamond}, we fill in the
matrix
$(s_{i,j})\in \alg F^{P\times Q}$ in such a way that
$s_{i,j}\preccurlyeq s_{i+1,j}$ and
$s_{i,j}\preccurlyeq s_{i,j+1}$ where they are defined.
Eventually, we get $s=s_{P,Q}$ such that $s\geq p,q$. Therefore
$s\geq p\geq x$ and $s\geq q\geq y$, so $x\sim y$.

Compatibility of $\sim$ with addition and multiplication is
straightforward. For $p_1,q_1,p_2,q_2\in\alg F$ such that $p_1\sim q_1$,
and $p_2\sim q_2$, there are $r_1, r_2$ such that
$p_1,q_1 \leq r_1$ and $p_2,q_2 \leq r_2$. Thus
$p_1+p_2 \leq r_1+r_2$ and $q_1+q_2 \leq r_1+r_2$.
Therefore
$p_1+p_2\sim q_1+q_2$, so $\sim$ is compatible with addition.

Regarding multiplication, consider any $p,q,s\in\alg F$ such that
$p\sim q$. There is $r$ such that $p,q\leq r$.
By Claim~\ref{prec-mul-comp} and $p\leq r$, we get
$sp, sq\leq sr$ and $ps,qs\leq rs$. Therefore $sp\sim sq$ and
$ps\sim qs$.

This is sufficient for compatibility with multiplication: If
$p_1\sim q_1$ and $p_2\sim q_2$,
then $p_1p_2\sim q_1p_2\sim q_1q_2$, so $p_1p_2\sim q_1q_2$ by
transitivity.
\end{proof}

\subsection{Decreasing $n$}

\begin{theorem}
\label{times-decrease}
Let $\alg A$ be an idempotent algebra with
$n\times(n+1)\times m$-terms for
some $n>1,m>0$. Then there exists $m'$ such that $\alg A$ has
$(n-1)\times n\times m'$-terms.
\end{theorem}
\begin{proof}
Without loss of generality, we can assume that the
$n\times(n+1)\times m$-terms 
$f_1,\ldots,f_{nm},g_1,\ldots,g_{(n+1)m}$ are the only basic
operations of $\alg A$, and $\alg A$ is free idempotent algebra generated by two
symbols $0$ and 1 modulo the equations describing the
$n\times(n+1)\times m$-terms.

Consider the subuniverse $R\leq \alg A^\omega$ generated by all the
infinite sequences that have the element $1$ at exactly one position
and the element $0$ everywhere else.

Notice that $R$ is invariant under all permutations of $\omega$ and since
$\alg A$ is idempotent, every sequence in $R$ has only finitely many
nonzero values.

By $\hat{\alg A}$ we denote the free commutative monoid generated by
all the non-zero elements of $\alg A$. We identify the element
$0\in\alg A$ with the neutral element in $\hat{\alg A}$.
For $\bar x\in R$, let $\hat x$
denote the sum of all nonzero values of $\bar x$, and let
$\hat R$ be the set $\{\hat x : \bar x\in R\}$.

\begin{claim}
\label{sufficient-sums}
To prove the theorem, it suffices to find
$$x_1,x_2,\ldots,x_{n-1},y_1,y_2,\ldots,y_n\in \hat R$$
such that
$x_1+\cdots+x_{n-1} = y_1+\cdots+y_n$.
\end{claim}
If that happens, we can choose large enough $m'$ and express the
elements $x_i,y_i\in\hat{\alg A}$ as follows:
$$
x_i = \sum_j^n\sum_k^{m'}z_{i,j,k} \hbox{ for any $i=1,\ldots,n-1$,}
$$$$
y_j = \sum_i^{n-1}\sum_k^{m'}z_{i,j,k} \hbox{ for any $j=1,\ldots,n$,}
$$
where $z_{i,j,k}\in\alg A$ for
$i=1,\ldots n-1,\, j=1,\ldots n, k=1,\ldots,m'$. Since elements $x_i$
are in $\hat R$, there are $(nm')$-ary terms $f'_i$ such that
if we put the element $1$ at the position $(j,k)$, and zeros
otherwise in $f_i$, we get $z_{i,j,k}$. Similarly, since elements $y_j$
are in $\hat R$, there are $((n-1)m')$-ary terms $g'_j$ such
that if we put $1$ at the position $(i,k)$ and zeros otherwise into
the term $g'_j$, we get $z_{i,j,k}$. So the equations of
$(n-1)\times n\times m'$-terms are satisfied by terms $f'_i,g'_j$ if
variables $x,y$ are substituted by $0$ and $1$, respectively.
Then the equations are satisfied in general, since $0,1$ are the
generators of the free algebra $\alg A$.

Every element of $\alg A$ is a binary function $t(0,1)$ on $A$ in
variables $0, 1$. We regard them as unary functions $t(1)$ where $0$ is a
constant and $1$ is the variable. With this viewpoint, there is a
multiplication on $A$ defined as usual function composition.
$(t_1t_2)(1) = t_1(t_2(1))$. This defines a structure of monoid on $A$
where $1$ is the neutral element and $0$ is an absorbing element.
For $i=1,\ldots,n,\,j=1,\ldots,(n+1),\,k=1,\ldots,m$, let
$b_{i,j,k}\in A$ be the element of the monoid defined by
$$
b_{i,j,k} = f_i(0,0,\ldots,0,\indexed 1{(j,k)},0,\ldots,0)
= g_j(0,0,\ldots,0,\indexed 1{(i,k)},0,\ldots,0),
$$
and let $\alg B$ be the submonoid generated the elements $b_{i,j,k}$.
Finally, let $\hat{\alg B} = (\hat B,+,\cdot,0,1)$ be the additive
submonoid of $\hat{\alg A}$ generated by elements of $\alg B$ with
multiplicative structure inherited from $\alg B$, so
$\hat{\alg B}$ is the free semiring generated by elements
$b_{i,j,k}$. Notice that the universe of $\hat{\alg B}$ is a subset of
the universe of $\hat{\alg A}$.

We equip the semiring $\hat{\alg B}$ with equations $E$ of the form
$$
\sum_i^n\sum_k^m b_{i,j,k}= 1 \hbox{ for all $j=1,\ldots,(n+1)$},
$$$$
\sum_j^{n+1}\sum_k^m b_{i,j,k}= 1 \hbox{ for all $i=1,\ldots,n$}.
$$
In other words, these equations actually say that
$$
f_i(1,0,\ldots,0)+f_i(0,1,\ldots,0)+\cdots+f_i(0,0,\ldots,1)= 1,
$$$$
g_i(1,0,\ldots,0)+g_i(0,1,\ldots,0)+\cdots+g_i(0,0,\ldots,1)= 1.
$$
Let $\sim$ be the congruence generated by these equations $E$.

\begin{claim}
If $p,q\in\hat{\alg B}$ such that $q$ is a single expansion of $p$
using equations $E$ and $p\in \hat R$, then also
$q \in \hat R$.
\end{claim}
Let $t$ be a term in $\omega$ variables (using just finitely many of
them) that takes the generators of $R$ and outputs some $\bar r\in R$ such
that $\hat r = p$. We prove the claim by induction on the
complexity of $t$. Let $p = uv+s$ and $q = uev+s$ where $u,v$ are
monomials, $s$ is a polynomial, and $e$ is a single expansion of $1$.
If $u=1$, we prove the claim directly.
Any single expansion $e$ of $1$ is of the form
$$
h(1,0,\ldots,0)+h(0,1,0,\ldots,0)+\cdots+h(0,\ldots,0,1),
$$
where $h$ is a basic operation of $\alg A$. Let us denote the arity of
$h$ as $k$ and the summands as $b_i$ for $i=1,\ldots,k$. So we can
write $e=\sum_{i=0}^k b_i$.
We take $k$ different
representations $\bar r_1,\ldots \bar r_k\in R$ that differs only in the
possition of $v$ (if there are multiple $v$ in $\bar r$, we vary the
position of one of them and fix the rest). Then
$h(\bar r_1,\ldots,\bar r_k)$ correspond to the polynomial $ev+s = q$.

If $u\neq1$, we use the induction hypothesis.
Assume that $\bar r = h(\bar r_1,\ldots,\bar r_k)$ for an elementary operation
$h$, where all the construction terms for $\bar r_1,\ldots,\bar r_k$ are
simpler. We follow the position of $uv$ in the sequence $\bar r$. On that
position, we see $uv = h(w_1,\ldots,w_k)$. There are two
possibilities. Either idempotency is applied and $w_1=\cdots=w_k$, or
one more letter is appended to the word, therefore all the elements
$w_i$ except one are zeros. In the case of idempotency, we use a
single expansion step to all the sequences $\bar r_i$ in the same way -- we
replace the position with $uv$ by multiple positions
covering $uev$. We denote these modified sequences $\bar r_i$
as $\bar r'_i$. The sequences $\bar r'_i$ were obtained from $\bar r_i$ using a
single expansion step, so they are in $R$ by induction hypothesis.
Finally, $\bar r'=h(\bar r_1,\ldots,\bar r_k)\in R$ and $q = \hat r'$.

In the other case, there is one non-zero $w_i = u_2v$, where
$u = u_1u_2$ and $u_1$ is one of the generators of $\alg B$. Again, we
replace the $u_2v$ in $R$ by $u_2ev$ in $\bar r_i$, getting $\bar r'_i\in R$ by
induction hypothesis. For $j\neq i$, we obtain $\bar r'_j$ just by
expanding the number of zeros at the position of $uv$ so that the
corresponding positions still match. Finally,
$\bar r'=h(\bar r'_1,\ldots,\bar r'_k)\in R$ and $q=\hat r'$.

\begin{claim}
To prove the theorem, it suffices to show that $n-1 \sim n$ in $\hat{\alg B}$.
\end{claim}
Indeed, if $n-1\sim n$, there is a common expansion $s$ by
Proposition~\ref{common-expansion}. Since $s$ is an expansion of
$n-1$, there are $x_1,\ldots,x_{n-1}$ such that
$\sum_i^{n-1} x_i=s$, and every $x_i$ is an expansion of
$1$. Similarly, since $s$ is an expansion of
$n$, there are $y_1,\ldots,y_{n-1}$ such that
$\sum_i^n y_i=s$, and every $y_i$ is an expansion of $1$.
Therefore all the elements $x_i,y_i\in \hat R$ and the assumptions of
Claim~\ref{sufficient-sums} are satisfied.

Now we translated the original problem into the language of the
semiring $\hat{\alg B}$ modulo $\sim$. Before general reasoning, we show the
idea on the example $n=2, m=1$. So $\hat{\alg B}$ is generated by
$b_{11}, b_{12}, b_{13}, b_{21}, b_{22}, b_{23}$, congruence $\sim$ is
generated by
$$
1\sim b_{11}+b_{12}+b_{13} \sim b_{21}+b_{22}+b_{23}
\sim b_{11}+b_{21} \sim b_{12}+b_{22} \sim b_{13}+b_{23},
$$
and we want to prove $1\sim 2$.
Clearly $2\sim 3$ since
$$
2 \sim (b_{11}+b_{12}+b_{13})+(b_{21}+b_{22}+b_{23})
= (b_{11}+b_{21})+(b_{12}+b_{22})+(b_{13}+b_{23}) \sim 3.
$$
Now, let us expand $1$ a bit.
$$
1 \sim b_{11}+b_{12}+b_{13} \sim b_{11}(b_{21}+b_{22}+b_{23}) +
(b_{11}+b_{12}+b_{13})b_{12} + (b_{11}+b_{12}+b_{13})b_{13}
$$$$
= b_{11}(b_{22}+b_{12}+b_{23}+b_{13}) + \cdots \sim 2b_{11}+\cdots
$$
We managed to get $2b_{11}$ in the expanded $1$. Since $2\sim 3$, we
get an extra $b_{11}$, and then collapse the expression using the
reverse process. Therefore $1\sim 1+b_{11}$. But there is nothing
special about the generator $b_{11}$, If we swapped
$b_{11}\leftrightarrow b_{21}$, $b_{12}\leftrightarrow b_{22}$,
$b_{13}\leftrightarrow b_{23}$, we would get $1\sim 1+b_{21}$ by the
same reasoning.
Therefore
$$
1 \sim 1+b_{21} \sim (1+b_{11})+b_{21} = 1+(b_{11}+b_{21}) \sim 2.
$$

Now, let us return to the general setup with generators
$b_{i,j,k}$ for $i=1,\ldots,n, j=1,\ldots(n+1), k=1,\ldots,m$,
and the congruence $\sim$ is generated by
$$
1\sim \sum_i^n\sum_k^m b_{i,j,k} \hbox{ for all $j=1,\ldots,(n+1)$},
$$$$
1\sim \sum_j^{n+1}\sum_k^m b_{i,j,k} \hbox{ for all $i=1,\ldots,n$}.
$$
From the equations, we derive $n\sim n+1$
$$
n\sim \sum_i^n\left(\sum_j^{n+1}\sum_k^m b_{i,j,k}\right)
= \sum_j^{n+1}\left(\sum_i^n\sum_k^m b_{i,j,k}\right)
= n+1.
$$
We fix $i',j',k'$. To prove that $(n-1)\sim(n-1)+b_{i',j',k'}$ it
suffices to get $nb_{i',j',k'}\sim(n+1)b_{i',j',k'}$ in an expanded form of $n-1$.

In the following calculations, by $x > y$ we mean $(\exists z: x=y+z)$.
$$
n-1 \sim (n-1)\sum_j^{n+1}\sum_k^m b_{i',j,k}
> (n-1)b_{i',j',k'} + \sum_{j\neq j'}^{n+1}\sum_k^m b_{i',j,k}
$$$$
= b_{i',j',k'}\cdot\sum_{i\neq i'}^n 1
  + 1\cdot\sum_{j\neq j'}^{n+1}\sum_k^m b_{i',j,k}
$$$$
\sim b_{i',j',k'}\left(\sum_{i\neq i'}^n\sum_j^{n+1}\sum_k^m b_{i,j,k}\right)
+ \left(\sum_j^{n+1}\sum_k^m b_{i',j,k} \right)
\left(\sum_{j\neq j'}^{n+1}\sum_k^m b_{i',j,k}\right)
$$$$
> b_{i',j',k'}\left(\sum_{i\neq i'}^n\sum_{j\neq j'}^{n+1}\sum_k^m b_{i,j,k}
+ \sum_{j\neq j'}^{n+1}\sum_k^m b_{i',j,k}\right)
$$$$
= b_{i',j',k'}\left(\sum_{j\neq j'}^{n+1}\sum_i^n\sum_k^m b_{i,j,k}\right)
\sim b_{i',j',k'}\cdot\sum^{n+1}_{j\neq j'}1 = nb_{i',j',k'}.
$$
Hence $n-1 \sim n-1 + b_{i,j,k}$ for any $i,j,k$.
We finaly get the desired congruence
$$
n-1 \sim n-1 + b_{1,1,1} \sim n-1 + b_{1,1,1} +
b_{1,2,1} \sim \cdots \sim n-1 + \sum_j^{n+1}\sum_k^m b_{1,j,k} \sim n.
$$

\end{proof}

\begin{corollary}
  \label{m-plus-m-weakest}
  Every $SD(\wedge)$ variety has $(m+m)$-terms for some $m$.
\end{corollary}

\section{Getting to $(3+m)$-terms}
\label{3-plus-m}

In this section, we prove the following

\begin{theorem}
Every $SD(\wedge)$ variety $\var V$ has a $(3+m')$-terms for large enough $m'$.
\end{theorem}
By Corollary~\ref{m-plus-m-weakest} we know that the variety has the
$(m+m)$-terms for some $m$, denote them $f, g_1, g_2$. For simplicity,
we may assume that the idempotent terms $f,g_1,g_2$ are the only basic
operations of the variety, and that they satisfy only the idempotence,
$(m+m)$-equations and their consequences. Let $\alg A$ be the
$\var V$-free algebra generated by elements $0,1$

Similarly as in the proof of Theorem~\ref{times-decrease},
we define $R_n$ to be a $n$-ary relation generated by tuples with
exactly one element 1 and zeros everywhere else,
where $n\in\{1,2,\ldots,\omega\}$.

For an algebra $\alg B\in\var V$, we
define a $\alg B$-\emph{pendant} to be any subuniverse
$P\subset \alg B\times\alg A^\omega$ that is invariant under
all permutations of the $\omega$ positions on $\alg A^\omega$.

For any $\alg B$-pendant $P$ we define
$P|_0,P|_1\leq \alg B$ as follows
$$
P|_0 = \{b\in\alg B: (b, (0,0,\ldots,0))\in P\}, \quad
P|_1 = \{b\in\alg B: \exists \bar r\in R_\omega\colon(b, \bar r)\in P\}.
$$
If $P|_0$ and $P|_1$ intersect, we call the pendant $P$ \emph{zipped}.
For a subuniverse $C\leq\alg B$ and an element $b\in\alg B$, let
$C[b]$ denote the smallest $\alg B$-pendant $P$ satisfying
$C\leq P|_0$ and $\{b\}\times R_\omega\leq P$. Therefore
$C=C[b]|_0$ and $b\in C[b]|_1$.
Clearly, if $b\in C$, the pendant $C[b]$ is zipped since
$b$ is contained in both $C[b]|_0$ and $C[b]|_1$.

\begin{claim}
\label{3+m-sufficient}
To prove the theorem, it suffices to show that the $\alg A^3$-pendant
$R_3[(0,0,0)]$ is zipped.
\end{claim}
Indeed, the pendant $P=R_3[(0,0,0)]$ is just $R_\omega$ viewed as a
subuniverse of $\alg A^3\times\alg A^\omega$. So when that pendant is
zipped, there is a common element $\bar r_3\in P|_0 = R_3$ and
$\bar r_3\in P|_1$. By expanding the definition of $P|_1$, we get
$\bar r_\omega \in R_\omega$ such that $(\bar r_3, \bar r_\omega)\in P = R_\omega$.
Let $g'_1$ be the term producing $\bar r_3$ from the generators of $R_3$,
$g'_2$ be the term producing $\bar r_\omega$ from the generators of $R_\omega$
and $f'$ be the term producing $(\bar r_3, \bar r_\omega)$ from the generators of
$R_\omega$. We can choose large enough $m'$ such that $g'_2$ uses at most
first $m'$ generators and $f'$ uses at most first $3+m'$ of them. So we
perceive $g'_2$ as $m'$-ary and $f'$ as $(3+m')$-ary. Since
\[
g_1\left(\begin{pmatrix}1\cr0\cr0\end{pmatrix}\begin{pmatrix}0\cr1\cr0\end{pmatrix}\begin{pmatrix}0\cr0\cr1\end{pmatrix}\right)=\bar r_3,\quad
g_2\left(\begin{pmatrix}1\cr0\cr\vdots\cr0\end{pmatrix}\cdots\begin{pmatrix}0\cr\vdots\cr0\cr1\end{pmatrix}\right)=\bar r_\omega,
\]
\[
f\left(\begin{pmatrix}1\cr0\cr\vdots\cr0\end{pmatrix}\cdots\begin{pmatrix}0\cr\vdots\cr0\cr1\end{pmatrix}\right)=
\begin{pmatrix}\bar r_3\cr\bar r_\omega\end{pmatrix},
\]
the equations of $(3+m')$-terms are satisfied when we plug in $x=0$
and $y=1$. However, the elements $0,1$ are the generators of a free
algebra, so the equations are satisfied in general.

\begin{lemma}
\label{wlog-gen}
Let $\alg B$ be an idempotent algebra, $C\leq\alg B$ its subuniverse,
$b\in\alg B$ an element and $P$ be a $\alg B$-pendant such that
$C\leq P|_0$ and $b\in P|_1$. Then $(C[b])|_1\leq P|_1$.
\end{lemma}
\begin{proof}
To see that, take an element $(b,\bar r_\omega)\in P$ such that
$\bar r_\omega\in R_\omega$. Let $\bar r_\omega$ be of the form
$(x_1, x_2, \ldots, x_n, 0, 0, \ldots)$ for some large
enough $n$. Since $P$ is invariant under permutations on $\alg A^\omega$, it
contains all the elements of the form
$$
(b, (0,0,\ldots,0,x_1,x_2,\ldots,x_n,0,0,\ldots))
$$
We construct a homomorphism
$\varphi\colon\alg A\to\alg A^n$ by mapping its generators
$$
0\mapsto(0,0,\ldots,0), 1\mapsto(x_1,\ldots,x_n).
$$
We naturally extend $\varphi$ to mapping
$\alg A^\omega \to (\alg A^n)^\omega = \alg A^\omega$. Notice that
$\varphi$ is an endomorphism of $R_\omega$ since it maps generators of
$R_\omega$ into $R_\omega$.

To finish the proof of the lemma, we take any $b'\in C[b] |_1$ and
show that $b'\in P|_1$. There
is $\bar r_\omega\in R_\omega$ such that $(b', \bar r_\omega)\in C[b]$. Then
$\varphi(\bar r_\omega)\in R_\omega$ and moreover
$(b', \varphi(\bar r_\omega))\in P$.
The latter holds since the endomorphism
$\psi\colon\alg B\times\alg A^\omega$ defined by
$\psi((y,x)) = (y,\varphi(x))$ maps the generators of $C[b]$ into $P$.
In particular $P$ contains all the elements
$\psi(b, (0,\ldots,0,1,0,\ldots))$
for any position of 1, and $\psi(c, (0,0,\ldots)$ for any
$c\in C$. So $b'\in P$ and this finishes the proof of the lemma.
\end{proof}

Now, let $h$ be the binary term defined as
$$
h(x,y) = f(\underbrace{\vphantom{y}xx\ldots x}_m,\underbrace{yy\ldots y}_m),
$$
\begin{lemma}
\label{h-absorbtion}
For any $\alg B$-pendant $P$ and $x\in P|_0, y\in P|_1$,
we have $h(x,y),h(y,x)\in P|_1$.
\end{lemma}
\begin{proof}
Without loss of generality, we may assume that
$(y, (1,0,\ldots,0))\in P$. If not, we use Lemma~\ref{wlog-gen}
and work with $(P|_0)[y]$ instead of $P$. Then the lemma follows from
the identities
\[
f\begin{pmatrix}
  x & x & \ldots & x, & y & y & \ldots & y \cr
  0 & 0 & \ldots & 0, & 1 & 0 & \ldots & 0 \cr
  0 & 0 & \ldots & 0, & 0 & 1 & \ldots & 0 \cr
   &  & \vdots &  & \vdots & \vdots & \ddots & \vdots \cr
  0 & 0 & \ldots & 0, & 0 & 0 & \ldots & 1 \cr
  0 & 0 & \ldots & 0, & 0 & 0 & \ldots & 0 \cr
    &   & \vdots &    &   &   & \vdots &  \cr
   \end{pmatrix} = 
 g_2\xmatrix{
    \multispan 4\hss$h(x,y)$\hss}{
    1 & 0 & \ldots & 0 \cr
    0 & 1 & \ldots & 0 \cr
\vdots&\vdots&\ddots&\vdots\cr
    0 & 0 & \ldots & 1 \cr
    0 & 0 & \ldots & 0 \cr
      &   & \vdots & \cr
  },
\]\[
f\begin{pmatrix}
  y & y & \ldots & y, & x & x & \ldots & x \cr
  1 & 0 & \ldots & 0, & 0 & 0 & \ldots & 0 \cr
  0 & 1 & \ldots & 0, & 0 & 0 & \ldots & 0 \cr
 \vdots & \vdots & \ddots & \vdots &   &   & \vdots &     \cr
  0 & 0 & \ldots & 1, & 0 & 0 & \ldots & 0 \cr
  0 & 0 & \ldots & 0, & 0 & 0 & \ldots & 0 \cr
    &   & \vdots &    &   &   & \vdots &  \cr
\end{pmatrix} =
 g_1\xmatrix{
    \multispan 4\hss$h(y,x)$\hss}{
    1 & 0 & \ldots & 0 \cr
    0 & 1 & \ldots & 0 \cr
\vdots&\vdots&\ddots&\vdots\cr
    0 & 0 & \ldots & 1 \cr
    0 & 0 & \ldots & 0 \cr
      &   & \vdots & \cr
  },
\]
The columns of the identities encode such sequences in $\alg B\times\alg A^\omega$ that
\begin{itemize}
\item are contained in $P$: This is apparent from the left hand side,
\item has elements $h(x,y), h(y,x)$ at their first coordinates,
\item the other part is contained in $R_\omega$: This is apparent from
the right hand side.
\end{itemize}
Therefore $h(x,y), h(y,h)\in P|_1$.
\end{proof}

\begin{lemma}
\label{butterfly}
Let $\alg B_1, \alg B_2$ be idempotent algebras, $P$ be a
$(\alg B_1\times\alg B_2)$-pendant. Assume that there exist
$x,y\in\alg B_1$ and $u,v\in\alg B_2$ such that
$(x,u),(y,u),(x,v)\in P|_0$ and $(y,v)\in P|_1$.
Then $P$ is zipped.
\end{lemma}
\begin{proof}
The pair $(h(x,y), h(v,u))$ is in the intersection
$P|_0\cap P|_1$. Indeed, it is contained in $P|_0$ since we can write
$$
\begin{pmatrix}h(x,y)\cr h(v,u)\end{pmatrix} = h\begin{pmatrix}x & y\cr v & u\end{pmatrix}.
$$
Alternatively, we can use the following expansion of $(h(x,y), h(v,u))$:
$$
\begin{pmatrix}h(x,y)\cr h(v,u)\end{pmatrix}
= h\left(
    h\begin{pmatrix}x, & x\cr v, & u\end{pmatrix},
    h\begin{pmatrix}y, & y\cr v, & u\end{pmatrix}
    \right).
$$
By Lemma~\ref{h-absorbtion} used twice, the pair is also an element of
$P|_1$, which completes the proof.
\end{proof}

\begin{lemma}
\label{zipped-equiv}
Let $\alg B_1, \alg B_2$ be idempotent algebras
and $R\leq\alg B_1\times\alg B_2$ be a compatible relation.
Assume that there are elements $x\in\alg B_1, u,v\in\alg B_2$ such
that $(x,u),(x,v)\in R$. Then for any $y\in\alg B_1$ the
$(\alg B_1\times \alg B_2)$-pendant $R[(y,u)]$ is zipped if and only
if the $(\alg B_1\times \alg B_2)$-pendant $R[(y,v)]$ is zipped.
\end{lemma}
\begin{proof}
It suffices to show the forward implication. Since $R[(y,u)]$ is
zipped, there is some $(y_0,u_0)\in R\cap R[(y,u)]|_1$. Consider the
4-ary relation
$$
R' = \{(a_1,a_1,a_2,a_2):(a_1,a_2)\in R\},
$$
and the $(\alg B_1^2\times \alg B_2^2)$-pendant
$P=R'[(x,y,u,v)]$. Since $(y_0,u_0)\in R[(y,u)]|_1$, we can find
a quadruple $(x_0,y_0,u_0,v_0)$ in $P|_1$ for some additionally
generated elements $x_0, v_0$. So
$$
(x_0,u_0)\in R[(x,u)]|_1, \quad (x_0,v_0)\in R[(x,v)]|_1,
\quad (y_0,v_0)\in R[(y,v)]|_1.
$$
Since $(x,u),(x,v) \in R$, also
$(x_0,u_0), (x_0,v_0)\in R$. Let $Q$ be the pendant $R[(y,v)]$.
We have $(x_0,u_0).(x_0,v_0).(y_0,u_0)\in Q|_0$ and
$(y_0,v_0)\in Q|_1$. Therefore, the pendant $Q$ is
zipped by Lemma~\ref{butterfly}.
\end{proof}

Finally, we define a relation $\ltimes$ on $\alg A$ as follows. We
write $x\ltimes y$ if there are $u,v\in R_3$ such that
\begin{itemize}
\item $(x,u,v)\in R_3$,
\item The $\alg A^3$-pendant $R_3[(y,u,v)]$ is zipped.
\end{itemize}
Notice that $\ltimes$ is reflexive: Indeed for any $x$, there are $u,v$ such
that $(x,u,v)\in R_3$. Then also $R[(x,u,v)]$ is zipped, so
$x\ltimes x$.

\begin{lemma}
\label{transfer-rel}
If $x\ltimes y$ and there are $c,x',y'\in\alg A$ such that
the triples $(x,c,x'), (y,c,y')$ are in $R_3$, then
$x'\ltimes y'$.
\end{lemma}
\begin{proof}
Consider $u,v$ as in the definition of the relation $\ltimes$.
We will show that $x'\ltimes y'$ by finding approptiate $u',v'$. We
set $u'=c, v'=y$, so the condition (ii) is satisfied since
$(y',c,y)\in R_3$ by symmetry of $R_3$. To establish $x\ltimes y$ we
need to prove that $R_3[(x',c,y)]$ is zipped, equivalently,
that $R_3[(y,c,x')]$ is zipped. We interpret $\alg A^3$ as
$\alg A\times\alg A^2$ and use Lemma~\ref{zipped-equiv}. We plug in
$$
x\mapsto x, y\mapsto y, u\mapsto (u,v), v\mapsto (c,x'),
$$
Indeed $(x,u,v), (x,c,x')\in R_3$ and
$R_3[(y,u,v)]|$ is zipped. So the assumptions of
Lemma~\ref{zipped-equiv} are satisfied, and consequently
$R[(y,c,x')]$ is zipped.
\end{proof}

We are finally ready to prove the theorem.
We start with $g_1(100\ldots 0)\ltimes f(100\ldots0)$ and get to
$1\ltimes h(1,0)$ using Lemma~\ref{transfer-rel} and the following
triples in $R_3$:
$$
\begin{pmatrix}g_1(10\ldots)\cr g_1(010\ldots)\cr g_1(0011\ldots)\end{pmatrix}
\begin{pmatrix}f(10\ldots)\cr f(010\ldots)\cr f(0011\ldots)\end{pmatrix},
\begin{pmatrix}g_1(0011\ldots)\cr 0\cr g_1(1100\ldots)\end{pmatrix}
\begin{pmatrix}f(0011\ldots)\cr 0\cr f(1100\ldots)\end{pmatrix},
$$$$
\begin{pmatrix}g_1(110\ldots)\cr g_1(0010\ldots)\cr g_1(0001\ldots)\end{pmatrix}
\begin{pmatrix}f(110\ldots)\cr f(0010\ldots)\cr f(0001\ldots)\end{pmatrix},
\ldots,
\begin{pmatrix}g_1(0\ldots0)\cr 0\cr g_1(1\ldots 1)\end{pmatrix}
\begin{pmatrix}h(0,1)\cr0\cr h(1,0)\end{pmatrix}.
$$

So, there are $u,v$ such that $(1,u,v)\in R_3$ and $R_3[(h(1,0),u,v)]$
is zipped. The first condition enforces $u=v=0$, so
$R_3[(h(1,0),0,0)]$ is zipped. However, by Lemma~\ref{h-absorbtion}
$(h(1,0),0,0)\in R_3[(0,0,0)] |_1$, so $R_3[(0,0,0)]$ is zipped
(Lemma~\ref{wlog-gen}, universality of pendant construction), and
the proof is finished by Claim~\ref{3+m-sufficient}.

\section{A counterexample for $(2+m)$-terms}
\label{counterexample}

Based on the result of the previous chapter that some $(3+m)$-terms
are
satisfied in every $SD(\wedge)$ variety, one could ask whether the
result could be strengthened to $(2+m)$-terms.
However, as we demonstrate in this section, such a generalization is
not possible.
Not only that there is an algebra in a $SD(\wedge)$ variety that does
not have $(2+m)$-terms but there is even such an algebra that belongs
to a congruence distributive variety.

Even stronger Maltsev condition than congruence distributivity is the
existence of a near unanimity term. A \emph{near unanimity} term (NU
term for short) is a term $t$ satisfying
$$
t(x,x,\ldots,x,\indexed yi,x,\ldots,x) = x
$$
for any position $i$.

There is no algebra having an NU term and no $(2+m)$-terms, since
putting $g_2$ to be the NU term and $f,g_1$ to be just the projections
on the first coordinate meet the requirements of the
$(2+m)$-terms. However, in our first example we demonstrate that one
existence of an NU term does not imply $(2+m)$-terms for a fixed $m$.

Consider the
following symmetric $n$-ary operations $t^{\alg A}_n, t^{\alg B}_n$
for $n\geq 5$ on rational numbers: Let $x_1\leq x_2\cdots\leq x_n$ be
a sorted input of such an operation. Then
$$
t^{\alg A}_n(x_1,\ldots,x_n) = \frac{x_2+\cdots+x_{n-1}}{n-2}, \quad
t^{\alg B}_n(x_1,\ldots,x_n) = \frac{x_3+\cdots+x_{n-2}}{n-4}.
$$
If the input is not sorted, we first sort it and then perform the
calculation. These operations are clearly NU, that is,
$$
t^{\alg A}_n(x,x,\ldots,x,y,x,\ldots x) = t^{\alg B}_n(x,x,\ldots,x,y,x,\ldots x) = x
$$
for any position of $y$.

For proving key properties of $t$, we need a lemma.
\begin{lemma}
\label{wlog-sorted}
Let $x_1,\ldots,x_n,y_1,\ldots,y_n\in\Q$ be such that
$x_i\leq y_i$ for all $i=1,\ldots,n$. Let $x_1',\ldots,x'_n$
be $x_1,x_2,\ldots,x_n$ sorted in increasing
order, and let $y'_1,\ldots,y'_n$ be sorted
$y_1,\ldots y_n$. Then $x'_i\leq y'_i$ for all $i$ and the set
$\{i : x'_i < y'_i\}$ is at least as large as the set
$\{i : x_i < y_i\}$.
\end{lemma}
\begin{proof}
Without loss of generality, let the numbers $x_i$ be increasing
in lexicographical order. Therefore $x_i=x'_i$ for all $i$. It is
possible to sort the sequence $y_i$ by consecutive application of
\emph{sorting transpositions}, that is swaping $y_i$ with $y_j$ if
$i<j$ and $y_i>y_j$. An example of such an process is the well known
bubble sort algorithm. We show that one sorting transposition
preserves the condition $x_i\leq y_i$ for all $i$, and does not shrink
the set $\{i : x_i < y_i\}$. In one such transposition, the swapped
positions $i,j$ are independent of all the others, so we may assume
that there are no others. In particular $n=2$, $x_1\leq x_2$,
$y_1> y_2$, $x_1\leq y_1$, $x_2\leq y_2$, $y'_1=y_2$, $y'_2=y_1$.
First $x_1\leq x_2\leq y_2$ and $x_2\leq y_2<y_1$, so $x_1\leq y'_1$
and $x_2 < y'_2$. This shows that $x_i\leq y_i$ for all $i$. Now, let
us investigate the number of strict inequalities. Since
$x_2 < y'_2$, the size of the set $\{i : x'_i<y'_i\}$ is at least 1. If
the size equals two, we are done. Otherwise $x_1=y'_1$, so
$x_1=x_2=y_2$. Since $x_2=y_2$, the size of the set
$\{i : x_i<y_i\}$ is at most one, so it is not larger than
$\{i : x'_i<y'_i\}$.
\end{proof}

\begin{claim}
\label{strict-t-compatibility}
For any
$x_1,\ldots,x_n,y_1,\ldots,y_n\in\Q$ such that $x_i\leq y_i$ for all
$i$, we have $t^{\alg A}_n(x_1,\ldots,x_n)
\leq t^{\alg A}_n(y_1,\ldots,y_n)$.
The inequality is strict if  $x_i < y_i$ for at least three $i$.
\end{claim}
Indeed, we can assume that $x_i$ and $y_i$ are sorted by
Lemma~\ref{wlog-sorted}. The first part is then clear from
definition of $t^{\alg A}$. If $x_i < y_i$ for at least three $i$, it
happens for at least one $i\neq 1,n$, and that $x_i<y_i$ causes the strict
inequality.

Consider the algebras
$\alg A_n = (\Q, t^{\alg A}_n)$, $\alg B_n = (\Q, t^{\alg B}_n)$.
For $m\geq 1$ define the sets
$U\subset\Q^2, V_m\subset\Q^m, W_m\subset \Q^{2+m}$ as follows:
$$
U = \{(a_1,a_2) : a_1+a_2=1\}, \quad
$$$$
V_m = \{(b_1,\ldots,b_m) : b_1\ldots b_m\geq 0\hbox{ and there is a nonzero $b_i$.}\}
$$
$$
W_m = \{(a_1,a_2,b_1,\ldots,b_m) :
$$$$
(a_1+a_2 < 1\hbox{ and }b_1\ldots b_m\geq 0)
\hbox{ or }(a_1+a_2 = 1\hbox{ and }b_1=\cdots=b_m=0))\}
$$

\begin{claim}
\label{claim-U}
For any $n\geq 5$, the set $U$ is a subuniverse of $\alg A_n^2$.
\end{claim}
The claim follows from the fact that if $x_1,x_2,\ldots,x_n$ is
non-decreasing, then also $1-x_n,\ldots,1-x_2,1-x_1$ is
non-decreasing.

\begin{claim}
\label{claim-V}
For any $n\geq 5$, $2m < n$ the set $V_m$ is a subuniverse of
$\alg B_n^m$.
\end{claim}
Indeed, if at least three of $x_1,\ldots,x_n$ are non-zero, then
$t^{\alg B}$ is also non-zero. Consider $m$-tuples
$\bar x_1, \bar x_2, \ldots, \bar x_n$. Every $m$-tuple
$\bar x_i$ has a non-zero position $p_i$. Since $2m < n$, one of
the positions has to repeat three times, $p=p_{i_1}=p_{i_2}=p_{i_3}$.
So the $m$-tuple $t(\bar x_1,\ldots \bar x_n)$ has a non-zero
element at the position $p$.

\begin{claim}
\label{claim-W}
For any $m\geq 1, n\geq 5$, the set $W_m$ is a subuniverse of
$\alg A_n^2\times\alg B_n^m$.
\end{claim}
For the same reason as in Claim~\ref{claim-U}, the projection of
$W_i$ to $\alg A^2$ is a subuniverse of $\alg A^2$. The question is
about subtle detail how it interacts with the $\alg B^m$-part.
Let us take $(2+m)$-tuples $\bar x_1,\ldots,\bar x_n\in W_m$ and show
that $t(\bar x_1,\ldots,\bar x_n)$ belongs to $W_m$ as well. Let
$a_{i,j}, b_{i,j}$ be matrices such that
$\bar x_j = (a_{1,j},a_{2,j},b_{1,j},\ldots,b_{m,j})$. We analyze two
cases:
\begin{enumerate}
\item For at most two columns $j$ it happens that
$a_{1,j}+a_{2,j} < 1$. Then all the other columns have zero
$\alg B^m$-part, so $t^{\alg B}(b_{i,1}) = 0$ for any $\alg B$-row
$i$. Hence $t(\bar x_1,\ldots,\bar x_n)\in W_m$.
\item For at least three columns $j$ it happens that
$a_{1,j}+a_{2,j} < 1$. In other words, at these three positions $j$ it
happens that $a_{1,j} < 1-a_{2,j}$ while non-strict inequality is
satisfied everywhere. Thus, by Claim~\ref{strict-t-compatibility}, we
have
$$
t^{\alg A}(a_{1,1},\ldots,a_{1,n}) < t^{\alg A}(1-a_{2,1},\ldots,1-a_{2,n}) =
1-t^{\alg A}(a_{2,1},\ldots,a_{2,n}).
$$
Equivalently,
$$
t^{\alg A}(a_{1,1},\ldots,a_{1,n}) + t^{\alg A}(a_{2,1},\ldots,a_{2,n}) < 1,
$$
so $t(\bar x_1,\ldots,\bar x_n)$ belongs to $W_i$.
\end{enumerate}
So, in both cases, the result belongs to $W_m$, and the claim is
established.

We are now ready to construct the counterexamples.

\begin{theorem}
For any $n,m$ such that $n\geq 5$ and $2m < n$,
there is an algebra having an $n$-ary NU-term, $n\geq 5$, but no $(2+m)$-terms.
\end{theorem}
\begin{proof}
The algebra is $\alg C_n = \alg A_n\times\alg B_n$. For a contradiction,
suppose that $\alg C_n$ has $(2+m)$-terms $f,g_1,g_2$. These terms are
common for all the algebras in the variety generated by $\alg C$. In
particular, there are operations $g^{\alg A}_1,g^{\alg A}_2,f^{\alg A}$
on $\alg A_n$ and 
$g^{\alg B}_1,g^{\alg B}_2,f^{\alg B}$ on $\alg B_n$ such that
\begin{align*}
g^{\alg A}_1(1,0) &= f^{\alg A}(1,0,0,0,0\ldots,0,0) = a_1,\cr
g^{\alg A}_1(0,1) &= f^{\alg A}(0,1,0,0,0\ldots,0,0) = a_2,\cr
g^{\alg B}_2(1,0,0,\ldots,0,0) &= f^{\alg B}(0,0,1,0,0\ldots,0,0) = b_1,\cr
g^{\alg B}_2(0,1,0,\ldots,0,0) &= f^{\alg B}(0,0,0,1,0\ldots,0,0) = b_2,\cr
&\vdots\cr
g^{\alg B}_2(0,0,0,\ldots,0,1) &= f^{\alg B}(0,0,0,0,0\ldots,0,1) = b_m.\cr
\end{align*}
The tuple $(a_1,a_2,b_1,\ldots,b_n)$ belongs to $W_m$ since $W_m$
contains all the columns on the right hand side. Similarly,
$(a_1,a_2)\in U$ and $(b_1,\ldots,b_m)\in V_m$ by left hand side.
But there are is no such tuple $W_m$ that is composed from the tuples in
$U$ and $V_m$.
\end{proof}

\begin{theorem}
There is an algebra in a congruence distributive variety that has no
$(2+m)$-terms.
\end{theorem}
\begin{proof}
The proof is similar, we take the algebra
$\alg C_6 = \alg A_6\times\alg B_6$. We just modify it a bit in order
to make $V_m$ a subuniverse for any $m$. Let $s$ be the following 4-ary
minor of $t$
$$
s(x,y,z,w) = t(x,y,z,w,w,w).
$$
Consider the algebra $\alg C' = (\Q^2, s^{\alg C})$. The algebra
$\alg C'$ is congruence distributive, since it has the following
directed J\'onsson terms written as minors of the term $s$:

\begin{align*}
s(xyzz) &= t(xyzzzz),\cr
s(xxyz) &= t(xxyzzz),\cr
s(zzyx) &= t(xxxyzz),\cr
s(zyxx) &= t(xxxxyz).\cr
\end{align*}

For the definition of directed J\'onsson terms, we refer the reader
to~\cite{Kazda2018}.

On the other hand, $\alg C'$ does not have any $(2+m)$-terms. For
a contradiction, let us assume that there are
term operations $f^{\alg C}, g_1^{\alg C},g_2^{\alg C}$ in the algebra
$\alg C$. So there are such terms even in
$\alg A'=(\Q,s^{\alg A})$ and $\alg B'=(\Q,s^{\alg B})$. We consider
the same $2+m$ equalities as in the previous proof, resulting in
$a_1,a_2,b_1,\ldots,b_m$. Since the basic operations of algebras
$\alg A',\alg B'$ are defined from the operations of the algebras
$\alg A,\alg B$, the set $U$ is still a subuniverse of $(\alg A')^2$ and
the set $W_m$ is still a
subuniverse of $(\alg A')^2\times(\alg B')^m$. So $(a_1,a_2)\in U$ and
$(a_1,a_2,b_1,\ldots,b_m)\in W_m$. We cannot directly use
Claim~\ref{claim-V} to ensure that $V_m$ is a subuniverse of
$(\alg B')^m$ since the claim assumes $2m<6$. However, it is still true.
We can check it manually: If $\bar x,\bar y,\bar z,\bar w\in V$ and
$w_i>0$ for some $i$, then even
$$
s^{\alg B}(x_i,y_i,z_i,w_i) = t^{\alg B}(x_i,y_i,z_i,w_i,w_i,w_i) > 0,
$$
so $s^{\alg B}(\bar x,\bar y,\bar z,\bar w)$ has a non-zero position.
Therefore $V_m$ is a subuniverse of $(\alg B')^m$,
$(b_1,\ldots,b_m)\in V_m$, and we get the same contradiction as
in the previous proof.
\end{proof}

\section{Further work}
\label{further-work}

Since Question~\ref{main-question} remained open, the main
objective is still to find out whether or not the $SD(\wedge)$
property is characterized by a strong Maltsev condition.
The $(3+n)$-terms are general enough for $SD(\wedge)$ while the
$(2+n)$-terms are too strong. Therefore we suggest $(3+3)$-terms as the
candidate for a strong Maltsev condition, or a good starting point
for proving the opposite.

\begin{question}
  Is there a $SD(\wedge)$ variety that does not have $(3+3)$-terms?
\end{question}

It is also reasonable to start
with a stronger property than congruence meet-semidistributivity,
namely simple congruence distributivity, or the one in
Theorem~\ref{wnu-glued}.

\begin{question}
  Are $(3+3)$-terms implied by
  \begin{enumerate}[(a)]
  \item directed J\'onsson terms? (equivalent to congruence
    distributivity, see~\cite{Kazda2018})
  \item ternary and 4-ary weak NU terms $w_3.w_4$ such that
    $w_3(y,x,x)=w_4(y,x,x,x)$?
  \end{enumerate}
\end{question}

Mikl\'os Mar\'oti with Ralph McKenzie (see Theorem 1.3 of
\cite{CD-to-WNU}) proved that congruence
distributivity implies the
existence of all at least ternary weak NU terms.
However, the catalogue of counterexamples is so weak, that even the
``glued'' weak NU terms, as in item (b), are still plausible
candidates
for the strong Maltsev condition too. On the other hand, congruence
distributivity is the weakest general condition under which we know
about the weak NU terms. So we ask the following.

\begin{question}
  Is the existence of a weak NU term implied by the $SD(\wedge)$ property?
  In particular, is it implied by $(3+3)$-terms?
\end{question}

\bibliographystyle{plain}
\bibliography{bib-file.bib}

\end{document}